\documentclass[11pt]{amsart}%{article}
\usepackage{color}
\usepackage{amssymb}
\usepackage{amsmath}
\theoremstyle{plain}
\newtheorem{Thm}{Theorem}

\newtheorem{Pro}[Thm]{Proposition}

\newtheorem{Def}[Thm]{Definition}

\newcommand{\re}{\mathbb{R}}

\begin{document}

%begin Topmatter
\title[Translating solitons]
{ Stability property and Dirichlet problem for translating solitons }

\author{Li Ma, Vicente Miquel}

\address{ Prof. Dr.Li Ma, \\
 School of Mathematics and Physics\\
  University of Science and Technology Beijing \\
  30 Xueyuan Road, Haidian District
  Beijing, 100083\\
  P.R. China
}
\email{lma17@ustb.edu.cn}

\address{ Vicente Miquel \\
Department of Geometry and Topology \\
University of Valencia \\
46100-Burjassot (Valencia), Spain
}

\email{miquel@uv.es}

\thanks{The research is partially supported by the National Natural Science
Foundation of China No.11771124 and by the
MINECO (Spain) and FEDER  project MTM2016-77093-P. \\
 This work was done when the first named
author was visiting Valencia University in July 2019 and he would
like to thank the hospitality of the Department of Mathematics.
}

\begin{abstract}
In this paper, we prove that the infimum of the mean curvature is zero
for a translating solitons of hypersurface in $\re^{n+k}$. We give some conditions under which a complete hypersurface translating soliton is stable. We show that if the norm of its mean curvature is less than one, then the weighted volume may have exponent growth. We also study the Dirichlet problem for graphic translating solitons in higher codimensions.


\end{abstract}

 \maketitle

\section{introduction}
The study of translating solitons to mean curvature flow is an important subject in the understanding of the singularity analysis of geometric analysis \cite{H} \cite{I} \cite{MSS} \cite{Ma} \cite{X}\cite{W}\cite{Sm}.
In this paper, we continue to study the translating solitons of properly immersed submanifolds \cite{MV}
$F=F(X,t)\subset \re^{n+k}$, $X\in M^n, 0\leq t<T$, evolving under the mean
curvature flow defined by
$$(\partial_tF)^{\perp}=\vec{H}(F),$$
where $\vec{H}(F)$ is the mean curvature vector of the submanifold
$F=F(X,t)$ at time $t$ and $M\subset \re^{n+k}$ is a fixed manifolds. We denote by
$$
M(t)=F(M,t),\ \text{ and by }\   V(t)=vol (M(t)),
$$
the volume of the surface $(M(t))$.

Recall that when $1\leq \alpha\leq k$, translating solitons are characterized by the soliton
equation
\begin{equation}\label{soliton}
H_\alpha=<\nu_\alpha,W>
\end{equation}
where $\nu_\alpha$ is an outer unit normal to the fixed surface $M\subset
\re^{n+k}$, $\vec{H}(F)=-H_\alpha\nu_\alpha$, and $W$ is a fixed unit vector in
$\re^{n+k}$.  When there is no confusion, we identify the position vector $X$ of $F(M^n)$ in $\re^{n+1}$. Here and below, we use the notations as in \cite{MV} and \cite{H} such that $\vec{H}=-H_\alpha\nu_\alpha$ is the mean curvature vector and $A=\{h_{\alpha ij}\}:=(h_\alpha)$ is the second fundamental form with $h_{\alpha ij}=<D_{e_i}\nu_\alpha,e_j>$ and $H_\alpha=h_{\alpha jj}$ for the moving orthonormal frame ${e_i}$ on $M$ and the moving normal orthonormal frame $\{\nu_\alpha\}$, where $D$ denotes the usual directional derivative in $\re^{n+k}$. Define the function $S(X)=<X, W>$. Then the soliton equation \eqref{soliton} can
be written as
\begin{equation}\label{HDS}
H_\alpha = D_{\nu_\alpha} S  \ \ \text{in}  \ \  M.
\end{equation}

We have the following property about translating solitons.

\begin{Pro}\label{valencia1}
 If $M$ is a complete translating soliton $M$  such that $\nabla^\bot \xi=0$ for every point where $\xi= -\vec{H}/|\vec{H}|$ is well defined  (in particular for $M$ a hypersurface) one has
$$
\inf_M |H|^2=0,
$$
which is equivalent to the relation
$$
\sup_M |\nabla S|^2=1,
$$
 where $\nabla S$ denotes the gradient of $S$ restricted to $M$.
\end{Pro}
 We conjecture that the property is true without the condition $\nabla^\bot\xi=0$.

 \begin{Pro}\label{valencia2}
 If $M$ is a complete translating soliton $M$, then for any point $X\in M$,
 $$
 S(X) >\inf_M S(X)=\underline{\lim}_{d(X,X_0)\to \infty} S(X),
 $$
 where $X_0\in M$ and $d(X,X_0)$ is the geodesic distance between $X$ and $X_0$ on $M$.
\end{Pro}

 When we look at the nicest examples of translating solitons, the grim reaper and the bowl, we see that $\sup_M H^2=1$. However, the tilted grim reapers give examples with all the possible values of $\sup_M H^2$ between $1$ and $0$. It is an interesting question to consider geometric and analytical properties about translating solitons $M$ with $\sup_M H^2<1$.

 To state our result, we set
 $$
 dm=dm_X= \exp(-S(X)) dv_{g}.
 $$
 We have the following weighted volume growth result.
\begin{Thm}\label{mono}
 Assume that $(M,g)$ is a complete  translating soliton in  $\re^{n+k}$. Assume that $\sup_M|\bar{H}|^2<1$. Then there exists a positive number $a_0>1$ and $\epsilon>0$ such that for all  $a>a_0$, $u(X)=\exp(aS(X))$ on $ M$,
 we have the exponent growth of the weighted volume,i.e.,
 $$
 \int_{M_R} u dm\geq \exp (\epsilon (R-R_0))\int_{M_{R_0}} u dm,
 $$
 where $M_R=B_R(0)\cap M$ with $R\geq R_0>0$.
\end{Thm}

We may ask which condition could make a hypersurface translating soliton stable in the sense that
for any $f\in C_0^2(M)$, we have
 $$
 \int_M(-\Delta f +\nabla_{W^T} f-|A|^2f)f\exp(-S(X))dv_g\geq0.
 $$
 Note that this stability condition is equivalent to the statement that for any compact smooth subdomain $\Omega\subset M$, the first eigenvalue $\lambda_1(\Omega,dm)$ defined by
 $$
 \lambda_1(\Omega,dm)=\inf\{\int_\Omega (|\nabla f|^2-|A|^2f^2)dm; f\in H^1_0(\Omega,dm),\ \int_\Omega f^2dm=1\}
 $$
 is non-negative.

 We have the following result.
\begin{Thm}\label{mali-VL2}
 Assume that $(M,g)$ is a complete hypersurface translating soliton in  $\re^{n+1}$.
Then under the condition that $a(a-1)-a^2H^2+|A|^2\leq 0$ on $M$ for some real constant $a>0$,
 $M$ is stable in the sense above.
\end{Thm}
The proof of this result will be given in section \ref{sect2}.
One example may be $a=1$ and $|A|^2\leq H^2$, which is always true for convex hypersurfaces.

\vspace{0.6cm}

We also have the following result.
\begin{Thm}\label{mali-VL3}
Assume that $(M,g)$ is a complete hypersurface translating soliton in  $\re^{n+1}$.
Assume that there is a direction $e$ such that $<\nabla H,e>\leq 0$ on $M$ and $<\nu, e>\geq $ on $M$
such that for some $p<0$,
$$
\frac{1}{R^2}\int_{T_R} <\nu, e>^{p+1}\to 0, \ \ as \ R\to \infty,
$$
where $T_R=M_{2R}\setminus M_R$.
Then $M$ is a hyperplane.
\end{Thm}

  Note that
 $$
 <\nabla H,e>=II(e^T,a^T).
 $$
   We remark that the assumption is natural in the case when $M=\{x+u(x)e_{n+1}; x\in \Omega\subset R^n\}$ is a graph of the function $u$ and we choose $e=e_{n+1}$.
 \vspace{0.8cm}

We also consider the Dirichlet problem for the translating mean curvature flow of properly immersed submanifolds
$F=F(X,t)\in \re^{n+k}$, $X\in M^n,$ $0\leq t<T$; that is, the submanifolds evolve under the translating mean
curvature flow defined by
\begin{equation}\label{TMCF}
(\partial_tF)^\perp=\vec{H}(F) + W^\perp, \ \
\end{equation}
where $\vec{H}(F)$ is the mean curvature vector of the surface
$F=F(X,t)$ at time $t$, $W\in \re^{n+k}$ is a fixed unit vector and $W^\perp$ is the normal part of it,
and $M=F(\Omega, 0)\subset \re^{n+k}$ is a fixed surface and $\Omega\subset \re^{n}$ is bounded domain with boundary $\partial \Omega$.
Given the immersed submanifold $F_0(x)=F(x, 0)$ for $x\in \Omega$.
The Dirichlet boundary data and the initial data are given by $F(X,0)=F_0(X)$ for $X\in \Omega$ and $F(X,t) = F_0(X)$ for $X\in \partial \Omega$, that we shall abbreviate as
$$F\mid_{\Omega, \partial \Omega}=F_0\mid_{\Omega,\partial \Omega}.$$
We remark that the short-time existence of the solutions to (\ref{TMCF}) with the Dirichlet data and initial data is well-known. The Dirichlet problem for the translating mean curvature flow of graph hypersurfaces in $\re^{n+1}$ has been studied in \cite{Wm} and \cite{Ma2}, where the global flow is established and convergence result is also obtained in the bounded convex domain in $\re^n$. The key point is the use of Bowl solitons as the comparison surfaces.
It is clear that the Dirichlet problem for the translating mean curvature flow of properly immersed surfaces in higher co-dimensions is a very difficult problem. We can make some results for the graphical translating mean curvature flow in higher co-dimensions.

We now  give a partial reason why this graphical case for translating solitons in  higher co-dimensions is a difficult problem to study.
Let $W=(0,w)\in 0\times \re^k$. Given a smooth mapping $\psi:\overline{\Omega}\to \re^k$.
In the form (\ref{TMCF}), the Dirichlet boundary data and the initial data of graphical translating mean curvature flow
 (i.e., $F(x,t)=(x,f(x,t))$ satisfies (\ref{TMCF}) are given by
$$F\mid_{\Omega, \partial \Omega}=I\times \psi\mid_{\Omega,\partial \Omega},$$ where $\psi:\Omega\to \re^k$ is the vector-valued function such that the mapping $I\times \psi:\Omega\to \re^{n+k}$ is the embedding given by
$I\times \psi(x)=(x,\psi(x))$ for $x\in \Omega$. When $W=0$, the system is reduced to the mean curvature flow and the graphical case had been treated by M.Wang in \cite{Wm}. Even in this special case, there are many examples of smooth boundary maps constructed by Lawson-Osserman \cite{LO} (see also \cite{JS} for related) such that the Dirichlet problem to minimal surface system with such boundary datum can not be solved.  As one can expected, the Dirichlet problem to graphical translating soliton system is difficult to study. If one consider graphical translating solitons as the minimal surfaces in the incomplete Riemannian manifold
$$(\re^{n+k},exp(-<W,X>/n) |d X|^2),$$
 there is no complete Hilbert space for the variational structure for the bounded domain $\Omega$(i.e. the Dirichlet functional )
$$
D(F)=\int_\Omega |DF|^2exp(-<W,F>)dX.
$$
This causes the problem can not be attacked by the standard variational method.
One may want to use some constraint conditions to make a closed convex set such that  variational methods may be applied. However, this idea could be worked out when we have some comparison principles to use and this is indeed true for each component of the graphical solitons and for heat flow method as showed below in this paper.

This is one of the reason for us to use the heat flow method to study the Dirichlet problem to graphical translating soliton system and it may work well for small convex domain with small boundary data up to $C^2$ norm. This point will be make clear in section \ref{sect3}.

After establishing the boundary gradient estimate for the graphical translating mean curvature flow, we may use the argument in the proof of Theorem A in \cite{Wm} to prove the following result.
\begin{Thm}\label{mali4} Given a bounded $C^2$ convex domain $\Omega$ in $\re^n$ with diameter $D$ and the mapping
$$
\psi:\overline{\Omega}\to \re^k
$$
such that
$$
8D(n\sup_\Omega |D^2\psi|+1)+\sqrt{2}\sup_{\partial \Omega}|D\psi|<1.
$$
Then the Dirichlet problem for $\psi_{\partial \Omega} $ to the graphical translating mean curvature flow (\ref{GTMCF}) has a global solution.
\end{Thm}

We make the following remark about the limit of the global graphical flow in Theorem \ref{mali4}. Formally along the translating mean curvature flow (\ref{TMCF}), we have
$$
V(t)-V(0)=-\int_0^t\int_M|\bar{H}+W^{\perp}|^2.
$$
Then for the graphical translating mean curvature flow with uniformly controlled Lipschitz norm we have the limiting Lipschitz graphical surface $f_\infty:\Omega\to \re^k$ with the boundary data $\psi$. However, the regularity of the limiting surface will not be considered in this paper.

Here is the plan of this paper. In section \ref{sect2}, we prove Proposition \ref{valencia1}, Theorem \ref{mono}, Theorems \ref{mali-VL2} and \ref{mali-VL3}. In section \ref{sect3}, we study the Dirichlet problem for the graphical translating mean curvature flows and prove Theorem \ref{mali4}.

\section{Properties of translating solitons} \label{sect2}
Recall that from \eqref{soliton} it is easy to compute the Hessian of $S$ on $M$. In fact, for $U,V$ are tangent vector fields on $M$ such that $\nabla_UV=0$ at $X\in M$, we have
\begin{align}\label{HesS}
\nabla^2 S(U,V)(X)&=<D_{U}V,W>\\
&= - h_\alpha (U,V) <\nu_\alpha ,W> \nonumber\\
&= -H_\alpha\ h_\alpha(U,V). \nonumber
\end{align}

Denote by $\Delta$ the Laplacian operator of the induced metric on $M$. Then we have
\begin{equation}\label{SH}
\Delta S(X)=- H_\alpha  H_\alpha=-|\vec{H}|^2.
\end{equation}
From \eqref{HesS} and \eqref{SH} it is clear that $\bar{H}=0$ if and only if $W=W^T$ is parallel on $M$.

From the definitions of $S$ and the gradient operator $\nabla$ on $M$,
\begin{align}\label{DS1}
 DS = W, \ \ \nabla S = (DS)^T = W^T, \ \ (DS)^\bot = W^\bot = \vec{H},\\ |\nabla S|^2+ |\vec{H}|^2  =|W^T|^2 + |W^\bot|^2 =|W|^2=1. \nonumber
\end{align}

We now prove Proposition \ref{valencia2}: We argue by contradiction. Assume that there is some point $p\in M$ such that
$S(p)=\inf_M S(X)$. Then at $p$, we have
$$
\Delta S(p)\geq 0, \ \ \nabla S(p)=0.
$$
By \eqref{SH} we know that
$$
\vec{H}(p)=0.
$$
However, by \eqref{DS1} we have at $p$,
$$
1=|\nabla S|^2+ |\vec{H}|^2 =0,
$$
which is impossible. This completes the proof of Proposition \ref{valencia2}.

For the proof of Proposition \ref{valencia1}, we need to apply the maximum principle from \cite{AAR}.

\begin{proof}[{\large Proof of Proposition \ref{valencia1}}]
 If there is some point in $M$ where $\vec{H}=0$, then the Proposition \ref{valencia1} is proved. Then, we can suppose  $\vec{H} \ne 0$ everywhere. Then, there is a unit normal vector field $\xi$ such that $W^\bot=\vec{H}= - H \xi$ and $H>0$. We can compute
\begin{align}\label{DeltaH}
\Delta H &= - \Delta <\vec{H}, \xi > \\
&= -\Delta <W, \xi > \nonumber\\
&= -e_ie_i <W, \xi > \nonumber\\
&= - e_i<W, D_{e_i}\xi > \nonumber \\
&= - e_i (h_\xi(e_i, W^\top)) - e_i < W^\bot, \nabla^\bot_{e_i}\xi> \nonumber \\
&= - e_i (h_\xi(e_i, W^\top)) + e_i < H \xi, \nabla^\bot_{e_i}\xi> \nonumber\\
&= - e_i < h(e_i, W^\top),\xi>\nonumber \\
&= -<\nabla_{e_i}(h)(e_i,W^\top), \xi> - <h(e_i,\nabla_{e_i}W^\top), \xi>\nonumber \\
& \qquad  - <h(e_i,W^\top), \nabla^\bot_{e_i}\xi> .\nonumber
\end{align}
By Codazzi equation we have
\begin{align*}
<\nabla_{e_i}(h)(e_i,W^\top), \xi>
&= <\nabla_{W^\top}(h)(e_i,e_i), \xi> \\
& = <\nabla^\bot_{W^\top}\vec{H}, \xi> \\
& = - W^\top H.
\end{align*}
Moreover,
\begin{align*}
0&= D_XW = D_XW^\top + D_XW^\bot \\
& = \nabla_XW^\top - h(X,W^\top) - XH\ \xi - H A_\xi X - H \nabla^\bot_X\xi,
\end{align*}
 from which we obtain
\begin{align*}
\nabla_XW^\top &= H A_\xi X
h(X,W^\top) \\
&= - XH\ \xi - H \nabla^\bot_X\xi
<h(e_i,\nabla_{e_i}W^\top), \xi> \\
&= H <h(e_i,A_\xi e_i), \xi> \\
&= H |A_\xi|^2
\end{align*}
By substitution of these expressions in the above expression for $\Delta H$, we obtain
\begin{align}\label{DH}
 \Delta H = - <W^\top, \nabla H>  - H |A_\xi|^2 + H |\nabla^\bot\xi|^2 .
\end{align}
If $|A_\xi|^2 - |\nabla^\bot \xi|^2 \ge \delta |A_\xi|^2$ for some $\delta>0$ (what always happens if  $\nabla^\bot \xi=0$, in particular in codimension $1$), if $\inf H = \varepsilon>0$, since $|A_\xi|^2 = <h(e_i,e_j),\xi>^2 \ge H^2/n$, then $|A_\xi|^2 - |\nabla^\bot \xi|^2 >\delta \varepsilon^2/n$, which plugged into \eqref{DH} gives
\begin{equation}\label{DHl}
\Delta H < - \langle W^\top, \nabla H\rangle\ <\  \langle W^\top, \nabla H\rangle\  < - \delta \varepsilon^3/n.
\end{equation}
Taking $1$ and $S$ as the functions $q$ and $\gamma$ in Theorem A in \cite{AAR} (recall that $S$ satisfies
$|\nabla S| = |W^\top|\le 1$ and $|\Delta S| = H^2 \le 1$), we can apply it to conclude that there is a sequence $\{x_k\}_{k=1}^\infty$ in $M$ satisfying
$$H(x_k)<\varepsilon+ 1/k,\quad  |\nabla H(x_k)| < 1/k, \quad \nabla^2 H >-1/k,$$
which is in contradiction with \eqref{DHl}. Then $\inf_M H=0$.
\end{proof}

We now prove Theorem \ref{mono}.
\begin{proof}
 Since $\sup_M|\vec{H}|^2<1$, we may choose $a_0>1$ large enough such that
$$
\sup_M |\vec{H}|^2<\frac{a_0-1}{a_0}.
$$
Fore $a>a_0$ we let $u(X)=\exp(aS(X))$ on $ M$.

 Let
 $$Lf = \Delta f -\nabla_{W^T} f=\exp(S(X))div(\exp(-S(X))\nabla f)
 $$ be the drifting Laplacian associated to the density $dm_X= \exp(-S(X)) dv_{g_X}$. Recall, for any smooth function $\varphi$ on $M$ and for $M_R=B_R(0)\cap M$, we have
\begin{equation}\label{star2}
\int_{M_R} L\varphi dm = \int_{\partial M_R}\nabla \varphi\cdot \nu dm
\end{equation}
where $\nu$ is the outward unit normal to $\partial M_R$.

Note that
$$
\nabla_iu=au\nabla_iS, \ \ \nabla_{W^T}u=au|W^T|^2,
$$
and
$$
\Delta u=au\Delta S+a^2u|\nabla S|^2=au(-|\vec{H}|^2+a|W^T|^2).
$$
By $|\vec{H}|^2+|W^T|^2=1$, we have
$$
-|\vec{H}|^2+(a-1)|W^T|^2=-a|\bar{H}|^2+a-1=a(\frac{a-1}{a}-|\vec{H}|^2)\geq \epsilon.
$$
Recall
 $\sup_M |\vec{H}|^2\leq \frac{a-1-\epsilon}{a}$ for some positive $\epsilon$. We then have
\begin{equation}\label{ineq}
 Lu=(\Delta -\nabla_{W^T})u\geq (a\epsilon)u, \ \ on \ \ M.
 \end{equation}

Set $f(R)=\int_{M_R} u dm$ for any $R>0$.
 Using \eqref{ineq} and applying \eqref{star2} to $\varphi=u$ we have
 $$
 (a\epsilon)f(R)\leq \int_{M_R} Ludm=\int_{\partial M_R}\nabla u\cdot \nu dm.
 $$
 Since
 $$
 \int_{\partial M_R}\nabla u\cdot \nu dm=a\int_{\partial M_R}u\nabla S\cdot \nu dm\leq af'(R),
 $$
 we then have
 $$
 f'(R)\geq \epsilon f(R).
 $$
 Hence, for any $R\geq R_0>0$,
 $$
 f(R)\geq f(R_0)\exp (\epsilon (R-R_0)).
 $$
 This completes the proof.
 \end{proof}

Using similar argument we can prove Theorem \ref{mali-VL2} below.
\begin{proof}
For any real number $a$, let
$$
u(X)=u_a(X)=exp(aS(X)), \ \ on \ \ M.
$$

Note that
$$
\nabla_iu=au\nabla_iS, \ \ \nabla_{W^T}u=au|W^T|^2,
$$
and
$$
\Delta u=au\Delta S+a^2u|\nabla S|^2=au(-H^2+a|W^T|^2).
$$
Then
$$
(\Delta +|A|^2-\nabla_{W^T})u=au\big(-H^2+(a-1)|W^T|^2\big)+|A|^2u:=Qu.
$$
By $H^2+|W^T|^2=1$, we have
$$
-H^2+(a-1)|W^T|^2=-aH^2+a-1=a(\frac{a-1}{a}-H^2).
$$
Then, since $a(a-1)-a^2H^2+|A|^2\leq 0$ for some $a$,
$$
Qu=[a(a-1)-a^2H^2+|A|^2]u\leq 0.
$$
By a similar argument of Barta's Theorem (see Lemma 1.36 in \cite{CM}), we know that
$M$ is stable.
In fact, for  any compact smooth subdomain $\Omega\subset M$, we take the eigen-function $\phi\in C^2_0(\Omega)$ corresponding to the first eigenvalue $\lambda_1(\Omega):=\lambda_1(\Omega,dm)$; i.e.,
$\phi>0$ in $\Omega$ satisfies the equation
$$
(\Delta +|A|^2-\nabla_{W^T})\phi=-\lambda_1(\Omega)\phi,\ \ in \ \Omega
$$
with the boundary condition that $\phi=0$ on $\partial\Omega$. Note that by the maximum principle, we have
$\partial_\nu \phi (x)<0$ for all $x\in \partial\Omega$, where $\nu$ is the outward unit normal to $\partial\Omega$. Then
\begin{align*}
\lambda_1(\Omega)\int_\Omega \phi u dm+\int_\Omega \phi Qu dm&= \int_\Omega \phi(\Delta u +|A|^
2u-\nabla_{W^T}u)dm\\
&- \int_\Omega u(\Delta+|A|^2-\nabla_{W^T}) \phi dm\\
&=\int_\Omega \phi(\Delta u -\nabla_{W^T}u)dm- \int_\Omega u(\Delta-\nabla_{W^T}) \phi dm\\
&=\int_{\partial \Omega} \phi\partial_\nu u - \int_{\partial \Omega} u\partial_\nu\phi \\
&=- \int_{\partial \Omega} u\partial_\nu\phi\geq 0.
\end{align*}
By this we have
$$
\lambda_1(\Omega)\int_\Omega \phi u dm\geq -\int_\Omega \phi Qu dm\geq 0.
$$
That is, $\lambda_1(\Omega,dm)\geq 0$ for any compact smooth subdomain $\Omega\subset M$.
Then we complete the proof.
 \end{proof}

In the rest of this section we now prove Theorem \ref{mali-VL3} below.

\begin{proof}
  Assume that $M$ is not a hyperplane.
 For the fixed unit vector $e\in S^n$, let $f=<\nu, e>$. Then $f$ is non-trivial.
Recall that we have the relation for the hypersurface $M$ (see the proof of Proposition 3 in \cite{Ma1}) (see also \cite{Simon} \cite{SSY}),
$$
\Delta \nu=\nabla_M H-|A|^2\nu.
$$
 Then we have
$$
\Delta f=<\nabla H,e>-|A|^2 f, \ \ in \ M.
$$

Multiplying both sides by $f^p\eta^2$ , where $p<0$ is the real number and $\eta$ is a cut-off function on $M$ and integrating, we have
\begin{equation}\label{VA1}
\int |A|^2 f^{p+1} \eta^2 -\int <\nabla H,e>f^p\eta^2=\int \nabla f\nabla (f^p\eta^2).
\end{equation}
Note that
\begin{equation}\label{VA2}
\int \nabla f\nabla (f^p\eta^2)=\int pf^{p-1}|\nabla f|^2\eta^2+2\int \eta f^p\nabla f\nabla \eta.
\end{equation}
By Cauchy-Schwartz inequality we have
\begin{equation}\label{VA3}
2|\int \eta f^p\nabla f\nabla \eta|\leq \frac{|p|}{2} \int f^{p-1}|\nabla f|^2\eta^2+\frac{2}{|p|} \int f^{p+1}|\nabla \eta|^2.
\end{equation}
Combining \eqref{VA1}, \eqref{VA2} with \eqref{VA3} we have
$$
\int |A|^2 f^{p+1} \eta^2 -\int <\nabla H,e>f^p\eta^2+\frac{|p|}{2} \int f^{p-1}|\nabla f|^2\eta^2\leq \frac{2}{|p|} \int f^{p+1}|\nabla \eta|^2.
$$
Note $<\nabla H,e>\leq 0$.
$$
|\nabla \eta|^2\leq 4/R^2.
$$
By assumption that $\frac{1}{R^2}\int_{M_R} f^{p+1}\to 0$ as $R\to\infty$, we know that
$$
\int |A|^2 f^{p+1}=0, \ \ \int f^{p-1}|\nabla f|^2=0.
$$
Hence, $A=0$ on $M$ and then $M$ is a hyperplane.
\end{proof}

 \section{Dirichlet problems of translating solitons} \label{sect3}

In this section, we consider the Dirichlet problem for the translating mean curvature flow of properly immersed submanifolds
$F=F(X,t)\subset \re^{n+k}$, $X\in M^n, 0\leq t<T$; that is, the submanifolds evolve under the translating mean
curvature flow defined by
$$
(\partial_tF)^\perp=\vec{H}(F)+W^\perp, \ \
$$
where $\vec{H}(F)$ is the mean curvature vector of the submanifold
$F=F(X,t)$ at time $t$, $W\in \re^{n+k}$ is a fixed unit vector and $W^\perp$ is the normal part of it,
and $M=F(\Omega, 0)\subset \re^{n+k}$ is a fixed surface and $\Omega\subset \re^{n}$ is bounded domain with boundary $\partial \Omega$.
Given the immersed submanifold $F_0(x)=F(x, 0)$ for $x\in \Omega$,
the Dirichlet boundary data and the initial data are given by
$$F\mid_{\Omega, \partial \Omega}=F_0\mid_{\Omega,\partial \Omega}.$$
Since the short-time existence of the solutions to (\ref{TMCF}) with the Dirichlet data and initial data is well-known, we just assume it here.

Recall that
$$
\vec{H}(F)=\Delta F=\frac{1}{\sqrt{g}}\frac{\partial}{\partial x^i} (\sqrt{g}g^{ij}\frac{\partial F}{\partial x^i}).
$$
Here
$$
g_{ij}=\frac{\partial F}{\partial x^i}\cdot \frac{\partial F}{\partial x^j}, \ \ g=det(g_{ij}),
$$
and
$$
(g^{ij})=(g_{ij})^{-1}.
$$

Define
$\bar{F}=F-tW$.
Then we have
\begin{equation}\label{MTMCF}
(\partial_t\bar{F})^\perp=\Delta(\bar{F}) \ \
\end{equation}
Then by the maximum principle to each component of $\bar{F}$,
$$
\sup_{\Omega\times [0,T)}|\bar{F}^\alpha(x,t)|\leq \sup_{\Omega\times 0\bigcup \partial\Omega\times [0,T]}|\bar{F}^\alpha(x,t)|,
$$
and applying this to $\sup_{\Omega\times [0,T)}|{F}^\alpha(x,t)| \le \sup_{\Omega\times [0,T)}|\bar{F}^\alpha(x,t)| + \sup_{\Omega\times [0,T)}|t\ W^\alpha|$ we obtain the following result,

\begin{Pro}\label{mali1} Given $T>0$.
Let $F:\Omega\times [0,T)\to \re^{n+k}$ be a solution to the system (\ref{TMCF}) with the initial-boundary data $F_0$. Then
$$
\sup_{\Omega\times [0,T)}|F^\alpha|\leq \sup_\Omega |F_0^\alpha|+2|W^\alpha|T.
$$
\end{Pro}

We now consider the graphical translating mean curvature flow.
Here we take $W=(0,w)\in 0\times \re^k$. For a given smooth mapping $\psi:\overline{\Omega}\to \re^k$,
 the Dirichlet boundary data and the initial data of graphical translating mean curvature flow
 (i.e., $F(x,t)=(x,f(x,t))$ satisfies (\ref{TMCF}) are given by
$$F\mid_{\Omega, \partial \Omega}=I\times \psi\mid_{\Omega,\partial \Omega}$$, where $\psi:\Omega\to \re^k$ is the vector-valued function such that the mapping $I\times \psi:\Omega\to \re^{n+k}$ is the embedding given by
$I\times \psi(x)=(x,\psi(x))$ for $x\in \Omega$.

Note that
$$
\frac{1}{\sqrt{g}}\frac{\partial}{\partial x^i} (\sqrt{g}g^{ij}\frac{\partial F}{\partial x^i})=
\frac{1}{\sqrt{g}}\frac{\partial}{\partial x^i} (\sqrt{g}g^{ij})\frac{\partial F}{\partial x^i}+g^{ij}\frac{\partial^2 F}{\partial x^i\partial x^j}
$$
and the basis of tangent vectors is generated by
$$
\frac{\partial F}{\partial x^i}=e_i+\frac{\partial f}{\partial x^i}
$$
for $F(x)=(x,f(x))$.
Then
$$
g_{ij}=\frac{\partial F}{\partial x^i}\cdot\frac{\partial F}{\partial x^j}=\delta_{ij}+\frac{\partial f}{\partial x^i}\cdot\frac{\partial f}{\partial x^j}
$$
and
$$
\bar{H}=(g^{ij}\frac{\partial^2 F}{\partial x^i\partial x^j})^\perp.
$$
So the evolution equation (\ref{TMCF}) can be written as
$$
(\partial_tF)^\perp=(g^{ij}\frac{\partial^2 F}{\partial x^i\partial x^j}+W)^\perp.
$$
Note also that the basis tangent vector has non-trivial component in $\re^n$ and the vector
$$
g^{ij}\frac{\partial^2 F}{\partial x^i\partial x^j}
$$
has trivial component in $\re^n$.
As in \cite{Wm}, using the fact that $\perp$ is isomorphism on $\re^k$, we know that
the Dirichlet problem for graphical translating mean curvature flow in the graph case (i.e., $F(x,t)=(x,f(x,t))$) can be written as
\begin{equation}\label{GTMCF}
\partial_tf=g^{ij}\frac{\partial^2 f}{\partial x^i\partial x^j}+w, \ \
\end{equation}
with the boundary condition
$$
f\mid_{\partial \Omega}=\psi\mid_{\partial \Omega}.
$$
and the initial data given by $f\mid_{ \Omega\times 0}=\psi$.
We introduce the following elliptic operator on the domain $\Omega$
$$
L=g^{ij}\frac{\partial^2 }{\partial x^i\partial x^j}
$$
and the diameter of the domain $\Omega$
$$
D=diam(\Omega).
$$

Define
$\bar{f}=f-tw$. Then we can write the graphical mean curvature flow as
$$
\bar{f}_t =g^{ij}\frac{\partial^2 \bar{f}}{\partial x^i\partial x^j} \ \  t>0.
$$
Then, again by the maximum principle to each component of $\bar{f}$, we obtain

\begin{Pro}\label{mali2} Given $T>0$.
Let $f:\Omega\times [0,T)\to \re^k$ be a solution to the system (\ref{GTMCF}) with the initial-boundary data $\psi=(\psi^\alpha)$. Then
$$
\sup_{\Omega\times [0,T)}|f^\alpha-tw^\alpha|\leq \max(\sup_\Omega |\psi^\alpha|,\sup_{\partial\Omega\times(0,T)} |\psi^\alpha-tw^\alpha|)
$$
and
$$
\sup_{\Omega\times [0,T)}|f^\alpha|\leq \sup_\Omega |\psi|^\alpha+2|w^\alpha|T.
$$
\end{Pro}

We now derive the boundary gradient estimate to a solution to the system (\ref{GTMCF}) with the initial-boundary data $\psi=(\psi^\alpha)$ defined on a bounded $C^2$ convex domain $\Omega$ in $\re^n$. We have the result below.

\begin{Thm}\label{mali3} Given $T>0$ and $(f(x,t))$ the graphical translating mean curvature flow (\ref{GTMCF}) on $\Omega\times [0,T)$ with the initial-boundary data $\psi=(\psi^\alpha)$. We have the following boundary gradient estimate estimate
$$
\sup_{\partial \Omega\times [0,T]}|Df|\leq 4D(1+\tau)(n\sup_\Omega |D^2\psi|+1)+\sqrt{2}\sup_{\partial \Omega}|D\psi|
$$
where
$\tau=\sup_{\Omega\times [0,T)}|Df|^2$ for $Df=(\partial f/\partial x^j)$ and $D=diam(\Omega)$.
\end{Thm}
\begin{proof}

Let $p\in \partial \Omega$ and let
$\Gamma$ be the supporting $n-2$ dimensional hyperplane at $p$. Let $d_\Gamma$ be the distance function to $\Gamma$ in $\re^n$.
For each $\alpha=1,...,k$, $\mu>0$, and $K>0$, consider the following comparison function defined in $\re^n$
$$
S(x,t)=\mu \log(1+Kd_\Gamma(x))-f^\alpha+\psi^\alpha.
$$
We compute that
$$
-L\log(1+Kd_\Gamma(x))=\frac{K}{1+Kd_\Gamma(x)}(-Ld_\Gamma(x))+\frac{K^2}{(1+Kd_\Gamma(x))^2}g^{ij}\frac{\partial d_\Gamma}{\partial x^i}\cdot\frac{\partial d_\Gamma}{\partial x^j}
$$
Note that $d_\gamma$ is linear,
$$
Ld_\gamma=0.
$$
Then we have
\begin{equation}\label{key}
(\partial_t-L)S(x,t)=\frac{\mu K^2}{(1+Kd_\Gamma(x))^2}g^{ij}\frac{\partial d_\Gamma}{\partial x^i}\cdot\frac{\partial d_\Gamma}{\partial x^j}-L\psi^\alpha+w^\alpha.
\end{equation}
Recall $\tau=\sup_{\Omega\times [0,T)}|Df|^2$ for $Df=(\partial f/\partial x^j)$.
Then on $\Omega\times [0,T]$,
$$
\frac{1}{1+\tau}I\leq (g^{ij})\leq I.
$$
By $|Dd_\Gamma|=1$  and for $x\in \Omega$,
$$
d_\Gamma(x)\leq D,
$$
we obtain
$$
\frac{1}{1+\tau}\leq g^{ij}\frac{\partial d_\Gamma}{\partial x^i}\cdot\frac{\partial d_\Gamma}{\partial x^j}\leq 1
$$
and
$$
\frac{\mu K^2}{(1+Kd_\Gamma(x))^2}g^{ij}\frac{\partial d_\Gamma}{\partial x^i}\cdot\frac{\partial d_\Gamma}{\partial x^j}
\geq \frac{\mu K^2}{(1+KD)^2}\frac{1}{1+\tau}
$$
Note that
$$
|-L\psi^\alpha+w^\alpha|\leq |-L\psi^\alpha|+|w^\alpha|\leq n|D^2\psi^\alpha|+|w^\alpha|.
$$

To make the right side of (\ref{key}), we require
\begin{equation}\label{key2}
\frac{\mu K^2}{(1+KD)^2}\frac{1}{1+\tau}\geq n|D^2\psi|+1\geq n|D^2\psi^\alpha|+|w^\alpha|.
\end{equation}
(by choosing $\mu>0$ large) and then we get
$$
(\partial_t-L)S(x,t)\geq 0, \ \ on \ \ \Omega\times [0,T).
$$
Note that on the boundary of $\Omega$, $S>0$ except $S(p,t)=0$ and at the initial time $t=0$,
$S(x,0)\geq 0$ on $\Omega$. Then by the maximum principle we know that
$$
S(x,t)> 0,  \ \ for  \ \ (x,t)\in \Omega\times (0,T).
$$
Applying the same process to the comparison function for each $\alpha=1,...,k$,
$$
\check{S}(x,t)=\mu \log(1+Kd_\Gamma(x))+f^\alpha-\psi^\alpha,
$$
we get
$$
\check{S}(x,t)<0,  \ \ for  \ \ (x,t)\in \Omega\times (0,T).
$$
Hence, at $(p,t)$, from the definition of the normal derivative of $f^\alpha-\psi^\alpha$, we have
$$
|\frac{\partial(f^\alpha-\psi^\alpha)}{\partial \nu}|(p,t)\leq \lim_{d_\Gamma(x)\to 0}\frac{\mu \log(1+Kd_\Gamma(x))}{d_\Gamma(x)}=\mu K,
$$
and then
$$
|\frac{\partial f}{\partial \nu}|(p,t)\leq |\frac{\partial \psi}{\partial \nu}|(p,t)+\mu K.
$$
By the Dirichlet data $f=\psi$ on the boundary we know that the tangent derivative of $f^\alpha-\psi^\alpha$ at $(p,t)$ is zero, which implies that
$$
|Df| \leq \mu K+\sqrt{2}|D\psi|
$$
on the boundary $\partial \Omega$. By minimizing $\mu K$ with the constraint (\ref{key2}) we know its minimum value is achieved at $ K=D^{-1}$ and
$$
\mu K=4D (1+\tau)(n\sup_\Omega |D^2\psi|+1).
$$

This completes the proof.
\end{proof}

Once we know Proposition \ref{mali2} and Theorem \ref{mali3}, we may use the same arguments in the proof of Theorem A in \cite{Wm} to prove the Theorem \ref{mali4}.

\end{document}